\documentclass[11pt]{amsart}

\usepackage{amscd,amssymb,amsopn,amsmath,amsthm,graphics,amsfonts,enumerate,verbatim,calc}
\usepackage[dvips]{graphicx}
\usepackage{url}

\usepackage{color}

\usepackage{amssymb,amsmath}

\headsep=1cm \numberwithin{equation}{section}
\newtheorem{theorem}{Theorem}[section]
\newtheorem{lemma}[theorem]{Lemma}
\newtheorem{proposition}[theorem]{Proposition}
\newtheorem{corollary}[theorem]{Corollary}
\newtheorem{claim}[theorem]{Claim}
\newtheorem{Maintheorem}{Main Theorem}

\newtheorem{Problem}[theorem]{Problem}

\theoremstyle{definition}
\newtheorem{definition}[theorem]{Definition}

\theoremstyle{remark}
\newtheorem{remark}[theorem]{Remark}
\newtheorem{example}[theorem]{Example}

\newcommand{\gr}{\operatorname{gr}}

\newcommand{\Spa}{\operatorname{Spa}}
\newcommand{\pr}{\operatorname{pr}}

\newcommand{\fm}{\frak{m}}
\newcommand{\fp}{\frak{p}}
\newcommand{\fq}{\frak{q}}

\newcommand{\plim}[1][]{\mathop{\varprojlim}\limits_{#1}}

\begin{document}
\title[Some ring-theoretic properties of rings via Frobenius and monoidal maps]
{Some ring-theoretic properties of rings via Frobenius and monoidal maps}

\author[K. Eto]{Kazufumi Eto}
\address{Department of Mathematics, Nippon Institute of Technology, Miyashiro, Saitama 345-8501, Japan}
\email{etou@nit.ac.jp.}

\author[J. Horiuchi]{Jun Horiuchi}
\address{Department of Mathematics, Nippon Institute of Technology, Miyashiro, Saitama 345-8501, Japan}
\email{jhoriuchi.math@gmail.com}

\author[K. Shimomoto]{Kazuma Shimomoto}
\address{Department of Mathematics, Tokyo Institute of Technology, 2-12-1 Ookayama, Meguro, Tokyo 152-8551, Japan}
\email{shimomotokazuma@gmail.com}

\thanks{2020 {\em Mathematics Subject Classification\/}: 13A18, 13B22, 13B35, 13G45}

\keywords{Complete integral closure, integral closure, monoidal map, perfectoid ring, valuation ring}


\begin{abstract}
The purpose of this note is to relate certain ring-theoretic properties of rings in mixed and positive characteristics that are related to each other by a tilting operation used in perfectoid geometry. To this aim, we exploit the multiplicative structure of the ``monoidal map", which is constructed on arbitrary $p$-adically complete rings.
\end{abstract}

\maketitle 
\tableofcontents

\section{Introduction}

Since perfectoid geometry was created by Scholze, Kedlaya-Liu and others, it has been a powerful tool for solving many outstanding problems in algebraic geometry, arithmetic geometry, and their neighboring areas. The foundation of perfectoid geometry is crucially built on exploiting the properties of a certain monoidal map. To continue the story in this introduction, let us introduce some important objects in our research. For any ring $A$ and a prime number $p>0$, we define the \textit{tilt} of $A$ as the limit
$$
A^{\flat}:=\plim[\rm{Frob}] \big\{\cdots \xrightarrow{\rm{Frob}} A/pA \xrightarrow{\rm{Frob}} A/pA\big\},
$$
where ${\rm{Frob}}:A/pA \to A/pA$ is the Frobenius map. Then $A^\flat$ is a ring of characteristic $p>0$ if $pA \ne A$. As already mentioned in the above, this plays a decisive role in deriving fundamental facts in perfectoid geometry. We will derive certain ring-theoretic properties of $A^\flat$ from $A$. To this aim, an important non-additive, but multiplicative map
\begin{equation}
\label{monoidmapconst}
\sharp:A^\flat \to A
\end{equation}
will be introduced for any $p$-adically complete ring $A$, through which one can transfer information from $A$ to $A^\flat$. This construction was originally introduced by Fontaine. Thus, let us call this map the \textit{monoidal map} (or \textit{Fontaine's sharp map}) throughout the present article. Before going further, we need the following condition.

\begin{enumerate}
\item[$(\rm{\bf{Perf}})$]: For an element $\varpi \in A$ and for any $n\geq 1$, there is a compatible system of $p$-power roots ${\varpi}^{\frac{1}{p^n}} \in A$.
\end{enumerate}

With this condition, we use the symbol ${\varpi}^{\flat}:=(\ldots,{\varpi}^{\frac{1}{p^2}}, {\varpi}^{\frac{1}{p}},\varpi)$ and it it easy to check that 
${\varpi}^{\flat}\in A^\flat$ from the definition. Let us recall the following result from \cite{Sch12}, which is fundamental in the geometry of perfectoid spaces.

\begin{theorem}[Kedlaya-Liu, Scholze]
\label{perfectoidspace}
Let $\mathcal{A}$ be a perfectoid algebra over a perfectoid field $K$ of characteristic $0$, let $\mathcal{A}^\circ$ be the set of powerbounded elements of $\mathcal{A}$ and let $A \subset \mathcal{A}^\circ$ be an open integrally closed subring. Moreover, denote by $A^\flat$ the tilt of $A$. Then the following statements hold:
\begin{enumerate}
\item
There is a topologically nilpotent element $\varpi \in K^\circ$ such that $p \in \varpi K^\circ$ and $\varpi^{\frac{1}{p^n}} \in K^\circ$ for all $n>0$, and $\mathcal{A}=A[\frac{1}{\varpi}]$. Set $\mathcal{A}^\flat:=A^{\flat}[\frac{1}{\varpi^\flat}]$ which is equipped with the structure as a complete Tate ring coming from that of $\mathcal{A}$.

\item
$\mathcal{A}^\circ$ (resp. $\mathcal{A}^{\flat\circ}$) is completely integrally closed in $\mathcal{A}$ (resp. in $\mathcal{A}^\flat$). Moreover, $A^{\flat}$ is an open integrally closed subring of $\mathcal{A}^{\flat\circ}$.

\item
$\mathcal{A}^\flat$ is a perfectoid ring of characteristic $p>0$ and $\mathcal{A}^{\circ\flat}=\mathcal{A}^{\flat\circ}$.

\item
There is a multiplicative map $\sharp:\mathcal{A}^\flat \to \mathcal{A}$ that induces a homeomorphism on adic spectra:
$$
\Spa(\mathcal{A},\mathcal{A}^\circ) \cong \Spa(\mathcal{A}^\flat,\mathcal{A}^{\circ\flat}).
$$
\end{enumerate}
\end{theorem}

The map $\sharp:\mathcal{A}^\flat \to \mathcal{A}$ is induced from $(\ref{monoidmapconst})$ in view of the facts that $\sharp(\varpi^\flat)=\varpi$ and $A[\frac{1}{\varpi}]=\mathcal{A}$ and $A^\flat[\frac{1}{\varpi^\flat}]=\mathcal{A}^\flat$. Let us explain the situation of the above theorem. First of all, notice that there is no obvious ring map connecting $A$ and $A^\flat$ directly. The element $\varpi \in A$ (resp. $\varpi^\flat \in A^{\flat}$) in Theorem \ref{perfectoidspace} gives a correspondence:

\begin{equation}
\label{specializationlift}
A \rightarrow A/\varpi A \cong A^{\flat}/\varpi^\flat A^{\flat} \leftarrow A^{\flat},
\end{equation}
where all maps are ring homomorphisms, and the right and left arrows are the natural quotient maps. Moreover, $A^{\flat}$ has positive characteristic. Quite generally, $A^\flat$ is known to be a perfect $\mathbb{F}_p$-algebra. It would be natural to consider the following problem.

\begin{Problem}
\label{integraltilting}
\begin{itemize}
\item
What kind of ring-theoretic properties on $A$ (resp. $A^\flat$) are carried over to $A^\flat$ (resp. $A$)?

\item
What kind of ring-theoretic properties can be derived via the monoidal map?
\end{itemize}
\end{Problem}

Our aim is to extend part of the content of Theorem \ref{perfectoidspace} to arbitrary $p$-adicaly complete rings that are not necessarily perfectoid. The following result derives complete integral closedness of $A^\flat$ from that of $A$.

\begin{Maintheorem}
\label{maincomplete1}
Fix a prime $p>0$ and let $A$ be a ring with a nonzero divisor $\varpi \in A$. Assume that $\varpi$ satisfies the condition $(\rm{\bf{Perf}})$. Then we have the following assertions:
\begin{enumerate}
\item
If $A$ is $p$-adically complete and completely integrally closed in $A[\frac{1}{\varpi}]$, then $A^{\flat}$ is also completely integrally closed in $ A^{\flat} [\frac{1}{\varpi^{\flat}}]$.

\item
If $A$ is completely integrally closed in $A[\frac{1}{\varpi}]$, and the $\varpi$-adic topology coincides with the $p$-adic topology on $A$, then $A^{\flat}$ is also completely integrally closed in $ A^{\flat} [\frac{1}{\varpi^{\flat}}]$.
\end{enumerate}
\end{Maintheorem}

A typical case where the first assertion of Main Theorem \ref{maincomplete1} applies is that $A=\widehat{B[t^{\frac{1}{p^\infty}}]}$ and $\varpi:=t$, where $B$ is some $p$-adic ring (such as the $p$-adic integer ring), $t$ is an indeterminate over $B$, and $\widehat{B[t^{\frac{1}{p^\infty}}]}$ is the $p$-adic completion of $B[t^{\frac{1}{p^\infty}}]$.\footnote{It is not clear whether $A$ is completely integrally closed in $A[\frac{1}{t}]$ or not, though we will not pursue this issue.} It is more subtle to prove integral closedness of $A^\flat$ in $A^\flat[\frac{1}{\varpi^\flat}]$, because the proof of Main Theorem \ref{maincomplete1} does not work directly. However, a trick using preuniromity introduced in \cite{NS19} can be used to deduce the following result.

\begin{Maintheorem}
\label{maincomplete2}
Let $A$ be a $p$-adically complete ring with a nonzero divisor $\varpi \in A$. Assume that $\varpi$ satisfies the condition $(\rm{\bf{Perf}})$, $p \in \varpi^p A$ and $A/pA$ is a semiperfect ring. If $A$ is integrally closed in $A[\frac{1}{\varpi}]$, then $A^{\flat}$ is also integrally closed in $ A^{\flat} [\frac{1}{\varpi^{\flat}}]$.
\end{Maintheorem}

Recall that $A/pA$ is \textit{semiperfect} if the Frobenius map on it is surjective. The authors are not sure if the assumption on semiperfectness can be dropped in Main Theorem \ref{maincomplete2}. We remark that the above results are essential in the construction of perfectoid almost Cohen-Macaulay algebras of some type in mixed characteristic. We refer the reader to \cite{IshizukaSh23} and \cite{NS23} for these topics.

$\textbf{Conventions}:$ In this article, all rings are assumed to be commutative with $1$. Rings are not necessarily assumed to be Noetherian. Because of this convention, we make a comment on the terminology of \textit{complete} rings/modules. We set $A$ as a ring and an ideal $I$ of $A$. We always assume that an $I$-adically complete module is always $I$-adically separated. In other words, we always have 
$$
\bigcap_{n\geq 0} I^{n} M = (0). 
$$
For generalities on the completion in the non-Noetherian setting, \cite[Tag 00M9]{Stacks} will be useful.

\section{Inverse perfection and the monoidal map}

Let us fix some notation. We write $p>0$ as a prime number.  First of all, we want to recall some definitions.

\begin{definition}
\label{completeintdef}
Let $A \to B$ be a ring extension. Then we say that an element $x \in B$ is \textit{almost integral} over $A$ if the $A$-submodule $\sum_{n=0}^\infty Ax^n \subset B$ is contained in a finitely generated $A$-submodule of $B$. We say that $A$ is \textit{completely integrally closed} in $B$ if every element of $B$ that is almost integral over $A$ belongs to $A$.
\end{definition}

Using this definition, the set of all elements in $B$ that are almost integral over $A$ forms a subring of $B$, which we call the \textit{complete integral closure} of $A$ in $B$. Denote this ring by $A_B^{\rm{ci}}$. We should be cautious about complete integral closure. It is not necessarily true that the complete integral closure $A_B^{\rm{ci}}$ is completely integrally closed in $B$. This observation was first made by Krull (see \cite{He69} for a modern treatment). Notice that when $A$ is Noetherian, then the notion of ``almost integral" coincides with the notion of ``integral". However, if we specialize to the case where $B=A[\frac{1}{t}]$ for a nonzero divisor $t \in A$, the situation is well-controlled. An element $x \in A[\frac{1}{t}]$ is almost integral over $A$ if there is some $c \in \mathbb{N}$ such that $t^cx^n \in A$ for all $n \ge 0$. Moreover, the complete integral closure of $A$ in $A[\frac{1}{t}]$ is completely integrally closed in $A[\frac{1}{t}]$ (see Corollary \ref{algcompleteint}).

Let $A$ be an $\mathbb{F}_p$-algebra. We say that $A$ is \textit{perfect} if the Frobenius map $F$ on $A$ is bijective. The following is immediate.

\begin{lemma}
\label{perfect.reduced}
Let $A$ be an $\mathbb{F}_p$-algebra. If $A$ is perfect, then $A$ is reduced.
\end{lemma}

\begin{proof}
Suppose the contrary. Let $x \in A$ be such that $x^m=0$ for some $m > 0$. Choose $k>0$ such that $m \le p^k$. Then we have $x^{p^k}=0$. Since $F^k(x)=x^{p^k}$ and the Frobenius map is injective on $A$, we have $x=0$, as desired.
\end{proof}

We have the following procedure to get the perfect ring.

\begin{definition}
Let $A$ be a ring. Then we define the \textit{tilt} (or \textit{inverse perfection} if $pA=0$) $A^{\flat}$ of $A$ as 
$$
A^{\flat}:=\plim[\rm{Frob}] \big\{\cdots \xrightarrow{\rm{Frob}} A/pA \xrightarrow{\rm{Frob}} A/pA\big\},
$$
where the transition map is the Frobenius map on $A/pA$.
\end{definition}

We write
$$
(a_n)_{n\geq 0}=(\ldots,a_1,a_0) \in A^\flat
$$
to ease notation.
One can check that $A^{\flat}$ is an $\mathbb{F}_p$-algebra with the zero $(\ldots,0,0)$ and the identity $(\ldots,1,1)$. We also notice that $A^{\flat}$ admits the natural ring map $\pr_0:A^{\flat} \to A/pA$ into its first ($0$-th) component.

\begin{lemma}
\label{PrimeFundLemma}
Let $p>0$ be a prime number and let $A$ be a ring. Assume that $t \in A$ satisfies $p \in tA$. Then for any $a,b \in A$ such that $a-b \in tA$, we have $a^{p^n}-b^{p^n} \in t^{n+1}A$.
\end{lemma}

\begin{proof}
This is an easy exercise.
\end{proof}

The next lemma is due to Fontaine. As we are unable to find references adequate for commutative algebraists, we decided give a proof.

\begin{lemma}[Monoid lemma]
\label{Monoid lemma}
Fix a prime number $p>0$. Assume that $A$ is a $p$-adically complete ring. Then there is an  isomorphism of multiplicative monoids
\begin{equation}
\label{monoidiso1}
\plim[x \mapsto x^p]A \cong A^{{\flat}},
\end{equation}
where the left side is the limit of the inverse system with each transition map defined by $x \mapsto x^p$, and the following properties are satisfied.

\begin{enumerate}
\item
$\plim[x \mapsto x^p]A$ can be equipped with a ring structure from $A^{\flat}$ via the monoidal isomorphism $(\ref{monoidiso1})$, where the addition is given by
\begin{equation}
\label{tiltaddition}
(a_n)_{n\geq 0}+(b_n)_{n\geq 0}=\Big(\lim_{m \to \infty} {(a_{n+m}+b_{n+m})}^{p^m}\Big)_{n\geq 0}
\end{equation}
and the multiplication is given by
\begin{equation}
\label{tiltproduct}
(a_n)_{n\geq 0}~(b_n)_{n\geq 0}=(a_n b_n)_{n\geq 0}.
\end{equation}

\item
Let $\plim[x \mapsto x^p]A$ and $A^{\flat}$ be equipped with the inverse limit topology, respectively, where $A$ has the $p$-adic topology and $A/pA$ has the discrete topology. Then $(\ref{monoidiso1})$ is an isomorphism of topological rings with respect to these topologies. 
\end{enumerate}
\end{lemma}

\begin{proof}
Let us construct the desired multiplicative map as follows. First of all, we have the commutative diagram of multiplicative monoids. 
$$
\begin{CD}
A @>>> A/pA \\
@VVV @VVV \\
A @>>> A/pA \\
\end{CD}
$$
Here, each horizontal map is the natural surjection and both vertical maps are given as the $p$-th power map. By taking inverse limits of these systems, we obtain the desired map $\plim[x \mapsto x^p]A \to A^{{\flat}}$. In order to prove that it is bijective, pick elements $(a_n)_{n\geq 0},(b_n)_{n\geq 0} \in \plim[x \mapsto x^p] A$ such that $a_n-b_n \in pA$ for all $n \ge 0$. Since $a_{n+k}^{p^k}=a_n$, $b_{n+k}^{p^k}=b_n$, and $a_{n+k}-b_{n+k} \in pA$, we obtain $a_n-b_n \in p^{k+1}A$ for any $k \ge 0$, which is independent of $n$ by Lemma \ref{PrimeFundLemma}. Since $A$ is $p$-adically separated, it follows that $a_n=b_n$ in $A$, which shows that $\plim[x \mapsto x^p] A \to A^{\flat}$ is injective.

Next, let us prove that $\plim[x \mapsto x^p] A \to A^{\flat}$ is surjective. Pick an element $(\overline{a_n})_{n\geq 0} \in A^\flat$. For any fixed $n \ge 0$, we write $a_n \in A$ as a lift of $\overline{a_n} \in A/pA$. Notice that $a_{n+k+1}^p-a_{n+k} \in pA$ and then this implies that the sequence
\begin{equation}
\label{Cauchyseq}
\{a_{n+k}^{p^k}\}_{k\geq 0}
\end{equation}
is Cauchy with respect to the $p$-adic topology, as follows from Lemma \ref{PrimeFundLemma}. Since $A$ is $p$-adically complete and separated, $\{a_{n+k}^{p^k}\}_{k\geq 0}$ admits a unique limit $b_n \in A$. In other words,
$$
(\ldots,b_1,b_0)=\Big(\ldots,\lim_{k \to \infty} a_{1+k}^{p^k},\lim_{k \to \infty} a_{0+k}^{p^k}\Big).
$$
It is easy to check that $(b_n)_{n\geq 0}$ gives an element of $\plim[x \mapsto x^p]A$, which is mapped to $(\overline{a_n})_{n\geq 0} \in A^\flat$. 

$(1)$: Let us verify addition and multiplication formula $(\ref{tiltaddition})$ and $(\ref{tiltproduct})$, respectively. Fix elements $(a_n)_{n\geq 0}, (b_n)_{n\geq 0} \in \plim[x \mapsto x^p]A$. Then the multiplication formula is obvious, so we prove the addition formula. Notice that every $a_n+b_n$ maps to $\overline{a_n}+\overline{b_n}$ in $A/pA$. While $(\overline{a_n}+\overline{b_n})_{n\geq 0}$ defines an element in $A^\flat$, it is not true that $(a_{n+1}+b_{n+1})^p=a_n+b_n$. In other words, $(a_n+b_n)_{n\geq 0}$ does not define an element of $\plim[x \mapsto x^p]A$. Therefore, it is necessary to modify the addition. But $(\ref{Cauchyseq})$ already shows that 
$$
\Big(\lim_{m \to \infty} {(a_{n+m}+b_{n+m})}^{p^m}\Big)_{n\geq 0}
$$
maps to $(\overline{a_n}+\overline{b_n})_{n\geq 0} \in A^\flat$. So we obtain the addition formula.

$(2)$: This follows from the construction of the monoid isomorphism $(\ref{monoidiso1})$.
\end{proof}

\begin{proposition}
Let $A$ be a ring. Let $p>0$ be a prime number and assume $A \ne pA$. Then the following assertions hold.
\begin{enumerate}
\item
$A^\flat$ is a perfect $\mathbb{F}_p$-algebra. In particular, it is always reduced.

\item
Let $A$ be a $p$-adically complete ring. Assume that $A$ is an integral domain. Then $A^{\flat}$ is also an integral domain.
\end{enumerate}
\end{proposition}

\begin{proof}
$(1)$: The second assertion follows from Lemma \ref{perfect.reduced}. We show that $A^\flat$ is a perfect $\mathbb{F}_p$-algebra. We write $(\overline{a_n})_{n\geq 0}=(\ldots, \overline{a_2}, \overline{a_1}, \overline{a_0}) \in A^\flat$, where $a_m \in A $ is a lift of $\overline{a_m} \in A/pA$ for all $m\geq 0$.
Let $F: A^{\flat} \to A^{\flat}$ be the Frobenius map on $A^{\flat}$, sending $(\overline{a_n})_{n\geq 0}=(\ldots,\overline{a_2}, \overline{a_1}, \overline{a_0}) \in A^\flat$ to $(\overline{{a_n}^p})_{n\geq 0}=(\ldots, \overline{{a_2}^p},\overline{{a_1}^p}, \overline{{a_0}^p}) \in A^\flat$. We prove that $F$ is bijective. Pick elements $(\overline{a_n})_{n\geq 0},(\overline{b_n})_{n\geq 0} \in A^{\flat}$ such that $(\overline{a_n})_{n\geq 0}\neq(\overline{b_n})_{n\geq 0}$. Thus, there exists at least one number $m$ such that $\overline{a_m} \neq \overline{b_m}$. Since $\overline{a_m}=\overline{{a_{m+1}}^p}$ and $\overline{b_m}=\overline{{b_{m+1}}^p}$, we have $\overline{{a_{m+1}}^p}  \neq \overline{{b_{m+1}}^p}$. Thus we get $F((\overline{a_n})_{n\geq 0})=(\ldots,\overline{{a_{m+1}}^p}, \overline{{a_m}^p},\ldots)$ and $F((\overline{b_n})_{n\geq 0})=(\ldots,\overline{{b_{m+1}}^p}, \overline{{b_m}^p}, \ldots)$. We conclude that $F((\overline {a_n})_{n\geq 0}) \neq F((\overline {b_n})_{n\geq 0})$. This proves the injectivity. To prove the surjectivity, fix an element $(\overline{a_n})_{n\geq 0}=(\ldots,\overline{a_2}, \overline{a_1}, \overline{a_0}) \in A^\flat$. Then $(\ldots,\overline{a_3}, \overline{a_2}, \overline{a_1})$ defines an element of $A^\flat$. We can check that $F((\ldots,\overline{a_3},\overline{a_2}, \overline{a_1}))=(\overline{a_n})_{n\geq 0}$, which proves the surjectivity.

$(2)$: Assume that $a, b \in A^\flat$ satisfy $ab=0$. The isomorphism $A^{\flat} \cong \plim[x \mapsto x^p]A$ allows us to write $a=(\ldots,a_2,a_1,a_0), b=(\ldots,b_2, b_1,b_0) \in A^\flat$, where $a_m,b_m \in A $ for all $m\geq 0$. Since $a_0b_0=0$ and $A$ is an integral domain, it follows that $a_0=0$ or $b_0=0$ holds. Assume that we have $a_0=0$. Since $a_m^{p^m}=a_0=0$ and $A$ is reduced, we have $a_i=0$ for all $i \ge 0$ giving $a=0$. If instead we assume $b_0=0$, we will get $b=0$, as desired.
\end{proof}

We discuss the key notion for the present article.

\begin{definition}[Monoidal map]
\label{Monoidalmap1}
Let $A$ be a $p$-adically complete ring. We define the \textit{monoidal map} $\sharp:A^{\flat} \to A$ as the composite map:
$$
\sharp :A^{\flat} \to \plim[x \mapsto x^p]A \xrightarrow{\pr_0}A,
$$
where the first map is the inverse of the isomorphism from Lemma \ref{Monoid lemma} and the the second map is the projection into the $0$-th component. Then we can check that $\sharp:A^{\flat} \to A$ is a continuous map with respect to the inverse limit topologies.
\end{definition}

From the construction of $\sharp :A^{\flat} \to A$, the following lemma is immediate.

\begin{lemma}
\label{WittEmbedding1}
Let $A$ be a $p$-adically complete ring. 
\begin{enumerate}
\item
$\sharp :A^{\flat} \to A$ is injective if and only if for any pair of compatible systems of $p$-power roots $(a_n)_{n \ge 0}$ and $(b_n)_{n \ge 0}$ such that $a_n,b_n \in A$ and $a_0=b_0$, we have $a_n=b_n$ for all $n \ge 0$.

\item
$\sharp :A^{\flat} \to A$ is surjective if and only if for any $x \in A$, we can find $y \in A$ such that $y^p=x$.
\end{enumerate}
\end{lemma}

For a certain $p$-torsion free ring, one can control $p$-th roots of a given element in a ring.

\begin{corollary}
\label{perfectTeich}
Assume that $A$ is a $p$-adically complete and $p$-torsion free ring such that $A/pA$ is a perfect $\mathbb{F}_p$-algebra. Then the equation $t^p-a=0$ has exactly one root in $\sharp(A^\flat)$ for any given element $a \in \sharp(A^\flat)$.
\end{corollary}

\begin{proof}
By Definition \ref{Monoidalmap1}, we have 
$\sharp :A^{\flat} \to \plim[x \mapsto x^p]A \xrightarrow{\pr_0}A$. By composing this with the natural quotient map $\pi: A \to A/pA$, we get $\pi \circ \sharp:A^\flat \to A/pA$. Since the Frobenius map on $A/pA$ is bijective, we get that $A^\flat \cong A/pA$ and hence, $\pi \circ \sharp$ coincides with the identity map on $A^\flat$. In particular, $\sharp$ is injective and it follows from Lemma \ref{WittEmbedding1} that the equation $t^p-a=0$ has a unique root in $\sharp(A^\flat)$.
\end{proof}

\begin{remark}
\label{WittEmbedding2}
\begin{enumerate}
\item
In the setting of Corollary \ref{perfectTeich}, it is not true that $x^p-a=0$ admits only one root in $A$ (rather than in $\sharp(A^\flat)$), even if $a \in \sharp(A^\flat)$. Let $p=2$ and consider $\mathbb{Z}_2$. Then the equation $t^2-1=0$ obviously has two roots $\pm 1$. However, $-1 \in \mathbb{Z}_2$ is not in the image of $\sharp:\mathbb{F}_2 \to \mathbb{Z}_2$.

\item
Corollary \ref{perfectTeich} explains that the map $\sharp$ is a natural generalization of the classical construction of the Teichm\"uller map for the Witt vectors of perfect rings. Assume that $A$ is $p$-torsion free and $p$-adically complete such that $A/pA$ is perfect. By the theory of Witt vectors, we have $A \cong W(A/pA)$. Moreover, we have an isomorphism $A^\flat \cong A/pA$, because the Frobenius on $A/pA$ is bijective. Then it can be checked that $\sharp:A^\flat \to A$ is identified with the Teichm\"uller map $\theta:A/pA \to W(A/pA)$.
\end{enumerate}
\end{remark}







\section{Proof of main theorems}

Let us define one terminology. We write $p>0$ as a prime number.

\begin{definition}
Let $A \to B$ be a ring extension. Then $A$ is \textit{$p$-root closed in $B$} if every element $b\in B$ satisfying $b^p \in A$ belongs to $A$. 
\end{definition}

\begin{lemma}
\label{extendmult}
Let $A$ be a $p$-adically complete ring with a nonzero divisor $\varpi \in A$ which satisfies the condition $(\bf{Perf})$. Assume that $A$ is $p$-root closed in $A[\frac{1}{\varpi}]$. Then the multiplicative map $A^\flat \to \plim[x \mapsto x^p] A$ extends to the multiplicative isomorphism $A^\flat[\frac{1}{\varpi^\flat}] \to \plim[x \mapsto x^p] (A[\frac{1}{\varpi}])$.
\end{lemma}

\begin{proof}
The multiplicative isomorphism $A^\flat \to \plim[x \mapsto x^p] A$ is refined to be a ring isomorphism via $(\ref{tiltaddition})$ and $(\ref{tiltproduct})$, which extends to
\begin{equation}
\label{MonoidExt}
A^\flat[\frac{1}{\varpi^\flat}] \cong (\plim[x \mapsto x^p] A)[\frac{1}{\varpi^\flat}] \to \plim[x \mapsto x^p] (A[\frac{1}{\varpi}]). 
\end{equation}
Then as $\varpi \in A$ (resp. $\varpi^\flat \in A^\flat$) is a nonzero divisor, it follows that the map $(\ref{MonoidExt})$ is injective. Next let $a=(\ldots,a_2,a_1,a_0) \in \plim[x \mapsto x^p] (A[\frac{1}{\varpi}])$ be any element. Then there is an element $c \in \mathbb{N}$ such that $\varpi^c a_0 \in A$. Then in $A[\frac{1}{\varpi}]$, we have
$$
(\varpi^{\frac{c}{p^{n+1}}} \cdot a_{n+1})^p=\varpi^{\frac{c}{p^{n}}} \cdot a_{n}
$$
for any $n \ge 0$. Then the $p$-root closedness of $A$ in $A[\frac{1}{\varpi}]$ implies that the element
$$
(\varpi^c)^\flat a=(\ldots,\varpi^{\frac{c}{p^2}} a_2,\varpi^{\frac{c}{p}}a_1,\varpi^ca_0)
$$
comes from $\plim[x \mapsto x^p] A$. In other words, $a \in  (\plim[x \mapsto x^p] A)[\frac{1}{\varpi^\flat}]$. This proves the surjectivity of $(\ref{MonoidExt})$.
\end{proof}

\begin{proof}[Proof of Main Theorem 1]
$(1)$: Notice that we do not necessarily assume $A/\varpi A$ to have characteristic $p>0$. By hypothesis, $\varpi$ admits all $p$-power roots in $A$ so that $\varpi^\flat \in A^\flat$. Suppose that $x \in A^\flat[\frac{1}{\varpi^\flat}]$ satisfies that $(\varpi^\flat)^c x^n \in A^\flat$ for some $c>0$ and all $n \in \mathbb{N}$. As the map $\sharp$ is multiplicative and $\sharp(\varpi^\flat)=\varpi$, we have
$$
\varpi^c \cdot \sharp(x)^n \in A.
$$
By complete integral closedness of $A$ in $A[\frac{1}{\varpi}]$, we have $\sharp(x) \in A$. We have a multiplicative map $\plim[x \mapsto x^p] (A[\frac{1}{\varpi}]) \cong A^\flat[\frac{1}{\varpi^\flat}]$ by Lemma \ref{extendmult}, because complete integral closedness implies that $A$ is $p$-root closed in $A[\frac{1}{\varpi}]$. After identifying $x$ as an element in $\plim[x \mapsto x^p] (A[\frac{1}{\varpi}])$ via the bijection $\plim[x \mapsto x^p] (A[\frac{1}{\varpi}]) \cong A^\flat[\frac{1}{\varpi^\flat}]$, we have $\sharp(x)=x_0$, where we write $x=(\ldots,x_2,x_1,x_0) \in \plim[x \mapsto x^p] (A[\frac{1}{\varpi}])$. Now $x_k^{p^k}=x_0 \in A$ and, so it follows from the $p$-root closedness that $x_k \in A$ for all $k \ge 0$. Hence $x \in A^\flat$, which finishes the proof of the first assertion.

$(2)$: Let $\widehat{A}$ be the $\varpi$-adic completion. Then since the $p$-adic topology is the same as the $\varpi$-adic topology, it follows that $\widehat{A}$ is $p$-adically complete. Then $\varpi$ is a nonzero divisor on $\widehat{A}$ and $\widehat{A}$ is completely integrally closed in $\widehat{A}[\frac{1}{\varpi}]$ in view of \cite[Corollary 2.7]{NS19}. Since $A^\flat$ is isomorphic to $\widehat{A}^\flat$, the proof of the second assertion is completed by applying the first assertion above.
\end{proof}

For the proof of the second main theorem, we need the following.

\begin{lemma}
\label{LemmaIntegralIdeal}
Let $A \to B$ be a ring extension with an ideal $I \subset A$ and assume that $I=IB$ in $B$; in other words, $I$ is an ideal in both $A$ and $B$. Then $A$ is integrally closed in $B$ if and only if $A/I$ is integrally closed in $B/IB$.
\end{lemma}

\begin{proof}
First, note that the natural map $A/I \to B/IB$ is injective by the assumption that $IB \subset A$ is an ideal of $A$, which coincides with $I$. Assume that $A$ is integrally closed in $B$. Let $\overline{b} \in B/IB$ be integral over $A/I$ and let $b \in B$ be the inverse image of $\overline{b}$ under $B \to B/IB$. Because $\sum_{n=0}^\infty (A/I) \overline{b^n}$ is a finitely generated $A/I$-submodule of $B/IB$, the assumption $IB \subset A$ implies that the inverse image of $\sum_{n=0}^\infty (A/I) \overline{b^n}$ under $B \to B/IB$ is $\sum_{n=0}^\infty A b^n$ which is a finitely generated $A$-submodule of $B$ and hence $b \in A$. So we get $\overline{b} \in A/I$, as desired. Conversely, let $b \in B$ be integral over $A$. Then $\sum_{n=0}^\infty (A/I) \overline{b^n}$ is a finitely generated $A/I$-submodule of $B/IB$ and hence $\overline{b} \in A/I$ by the integral closedness of $A/I$ in $B/IB$. Again by $IB \subset A$, we obtain $b \in A$.
\end{proof}

\begin{example}
Let us give an example which illustrates the situation of Lemma \ref{LemmaIntegralIdeal}. Namely, we construct an example of a ring extension $A \to B$ with an ideal $J \subset A$ such that $A \to B$ is not integral and $I:=JB \subset A$. Let $A$ be a valuation ring of dimension $2$ containing $\mathbb{F}_p$, and let $0 \subsetneq \fp \subsetneq \fq$ be the saturated chain of prime ideals of $A$, that is, $\fq$ is the maximal ideal. Choose a nonzero element $t \in \fp$ (called a ``pseudouniformizer"). Let $K$ be the field of fractions of $A$. In fact, we have $K=A[\frac{1}{t}]$. Let $A_\infty:=\bigcup_{n>0} A^{\frac{1}{p^n}}$ with the field of fractions $K_\infty=A_\infty[\frac{1}{t}]$. Then $A_\infty$ is the perfect closure of $A$, which is a perfect valuation ring of dimension $2$. Indeed, we have a saturated chain of prime ideals in $A_\infty$:
$$
0 \subsetneq \fp_\infty=\bigcup_{n>0} \fp^{\frac{1}{p^n}} \subsetneq \fq_\infty=\bigcup_{n>0} \fq^{\frac{1}{p^n}}.
$$
It is well known that the valuation ring is integrally closed in the field of fractions. Take the complete integral closure of $A_\infty$ in $K_\infty$. This is equal to the localization $(A_\infty)_{\fp_\infty}=(A_\fp)_\infty$ by a theorem of Krull (see Proposition  \ref{completeintKrull}). Let $J=tA_\infty$. Then we claim that the ring extension
$$
A_\infty \to (A_\fp)_\infty
$$
satisfies the hypothesis of Lemma \ref{LemmaIntegralIdeal}. Pick $x \in (A_\fp)_\infty$. Then since $x \in K_\infty$ is almost integral over $A_\infty$, there is an integer $m>0$ such that $t^m x^n \in A_\infty$ for all $n>0$ by the definition of almost integrality. By taking a sufficiently large $n>0$, we may assume $t^{p^n}x^{p^n} \in A_\infty$. The perfectness of $A_\infty$ then implies that
$$
tx=(t^{p^n}x^{p^n})^{\frac{1}{p^n}} \in A_\infty.
$$
So we conclude that $t(A_\fp)_\infty \subset A_\infty$.

One can modify the above construction to obtain an example $A \to B$ such that $A$ is not integrally closed in $B$. For example, one can construct a perfect domain $A$ that is not integrally closed in $B$, which is left as an exercise.
\end{example}

\begin{proof}[Proof of Main Theorem 2]
Let $B$ be the complete integral closure of $A$ in $A[\frac{1}{\varpi}]$.
First, notice that the inclusion $\varpi B \subset A$ holds in view of \cite[Lemma 2.3]{NS19}, which gives $pB \subset A$ because $p \in \varpi^pA$ by the hypothesis. Next consider an element $\varpi^{\frac{1}{p^k}} \cdot b \in B$ for arbitrary $b \in B$ and $k \in \mathbb{N}$. So we have
$$
(\varpi^{\frac{1}{p^k}} \cdot b)^{p^k}=\varpi \cdot b^{p^k} \in A.
$$
Since $A$ is assumed to be integrally closed in $A[\frac{1}{\varpi}]$, it is $p$-root closed in $A[\frac{1}{\varpi}]$, so it follows that $\varpi^{\frac{1}{p^k}} \cdot b \in A$. As $b \in B$ and $k \in \mathbb{N}$ are arbitrary, we have in particular
\begin{equation}
\label{tiltingcomplete}
\varpi^{\frac{1}{p^k}}B \subset A \subset B~\mbox{for every}~k>0.
\end{equation}

We need the following claim.

\begin{claim}
\label{completeintegral}
$B$ is completely integrally closed in $A[\frac{1}{\varpi}]=B[\frac{1}{\varpi}]$.
\end{claim}

\begin{proof}[Proof of Claim \ref{completeintegral}]
Let $b \in A[\frac{1}{\varpi}]$ be almost integral over $B$. Then there is an integer $k>0$ such that $\varpi^k b^n \in B$ for all $n>0$. It follows from \cite[Lemma 2.3]{NS19} that $\varpi^{k+1} b^n \in A$ for $n>0$, which implies that $b$ is almost integral over $A$. Thus, $b \in B$ and $B$ is completely integrally closed in $A[\frac{1}{\varpi}]$.
\end{proof}

Next we prove that $B$ is $p$-adically complete. Since $\bigcap_{n>0}p^{n+1}B \subset \bigcap_{n>0}p^{n}A=0$, it follows that $B$ is $p$-adically separated. Take a Cauchy sequence $x_n=\sum_{i=0}^nb_i p^i$ with $b_i \in B$. After truncation, it suffices to show that
$$
x'_n=\sum_{i=0}^nb_{i+1} p^{i+1}
$$
converges in $B$. Letting $c_i:=b_{i+1}p$, we have $x'_n=\sum_{i=0}^nc_i p^i$ with $c_i \in A$. Since $A$ is $p$-adically complete, the sequence $\{x'_n\}_{n \ge 0}$ converges to $x'_\infty \in A \subset B$. Now it follows from Claim \ref{completeintegral} and Main Theorem \ref{maincomplete1} that $B^\flat$ is completely integrally closed in $B^\flat[\frac{1}{\varpi^\flat}]$. It follows from the description of the tilt as a multiplicative monoid as in Lemma \ref{Monoid lemma} that the injection of $p$-adically complete rings $A \hookrightarrow B$ induces an injection $A^\flat \hookrightarrow B^\flat$. 

So $(\ref{tiltingcomplete})$ gives that $\varpi^\flat B^\flat \subset A^\flat$ and $A^\flat[\frac{1}{\varpi^\flat}]=B^\flat[\frac{1}{\varpi^\flat}]$ and in particular, $\varpi B$ is an ideal of $A$ and $\varpi^\flat B^\flat$ is an ideal of $A^\flat$. We obtain the commutative diagram:
$$
\begin{CD}
B^\flat @>\pi_B^\flat>> B^\flat/\varpi^\flat B^\flat @>\cong>> B/\varpi B @<\pi_B<< B\\
@AAA @AAA @AAA @AAA \\
A^\flat @>\pi_A^\flat>> A^\flat/\varpi^\flat B^\flat @>\cong>> A/\varpi B @<\pi_A<< A\\
\end{CD}
$$
Here, all vertical maps are injective which are induced by the injection $A \hookrightarrow B$, and $\pi_A$ and $\pi_A^\flat$ are the quotient maps. We claim that the isomorphism $A^\flat/\varpi^\flat B^\flat \cong A/\varpi B$ is induced from
\begin{equation}
\label{tiltisomor}
A^\flat=\varprojlim\{\cdots \xrightarrow{F}A/pA \xrightarrow{F} A/pA\} \to A/pA \to A/\varpi B,
\end{equation}
where the first map is the projection into the first component, and the second one is the quotient map. Likewise, $B^\flat/\varpi^\flat B^\flat \cong B/\varpi B$ is induced from the composite map
\begin{equation}
\label{tiltisomor2}
B^\flat=\varprojlim\{\cdots \xrightarrow{F}B/pB \xrightarrow{F} B/pB\} \to B/pB \to B/\varpi B.
\end{equation}
Let us first explain the isomorphism: $B^\flat/\varpi^\flat B^\flat \cong B/\varpi B$. We check that this is induced by $(\ref{tiltisomor2})$ as follows. Since $A/pA$ is semiperfect, it follows by \cite[Lemma 3.7]{NS19} that $B/pB$ is also semiperfect.
So the first map of $(\ref{tiltisomor2})$ is surjective and it follows that $B^\flat \to B/\varpi B$ is surjective. Since $B$ is integrally closed in $B[\frac{1}{\varpi}]$ and there is a surjection $B/pB \twoheadrightarrow B/\varpi B$, we find that the Frobenius map $F:B/\varpi B \to B/\varpi B$ is surjective and induces an isomorphism $B/\varpi^{\frac{1}{p^{n+1}}}B \cong B/\varpi^{\frac{1}{p^{n}}}B$ for any $n > 0$. Hence $(\ref{tiltisomor2})$ induces an isomorphism $B^\flat/\varpi^\flat B^\flat \cong B/\varpi B$, as desired. The same discussion using $(\ref{tiltisomor})$ yields an isomorphism: $A^\flat/\varpi^\flat B^\flat \cong A/\varpi B$.

Now we apply Lemma \ref{LemmaIntegralIdeal}. First, $A/\varpi B$ is integrally closed in $B/\varpi B$, so $A^\flat/\varpi^\flat B^\flat$ is integrally closed in $B^\flat/\varpi^\flat B^\flat$. Lemma \ref{LemmaIntegralIdeal} again applies to give that $A^\flat$ is integrally closed in $B^\flat$, which completes the proof.
\end{proof}

\begin{remark}
If $A/pA$ is semiperfect, it immediately implies by the Mittag-Leffler condition that
$$
{\mathbf{R}}^1 \varprojlim\{\cdots \xrightarrow{F} A/pA \xrightarrow{F} A/pA\}=0.
$$
In order to generalize Main Theorem \ref{maincomplete2} without the assumption of semiperfectness, it seems that the vanishing of the derived limits is an important problem to study. See \cite[Definition 3.5.6]{Wei94} for details.
\end{remark}

Let us record the following algebraic result.

\begin{corollary}
\label{algcompleteint}
Let $A$ be a ring with a nonzero divisor $\varpi \in A$ and let $B$ be the complete integral closure of $A$ in $A[\frac{1}{\varpi}]$. Then $B$ is completely integrally closed in $A[\frac{1}{\varpi}]$.
\end{corollary}

In the present article, we emphasized that the definition of integrality uses both additive and multiplicative structures of rings, while the definition of almost integrality involves the multiplicative structure of rings (see the paragraph just after Definition \ref{completeintdef}). Using this fact, we make the following observation.

\begin{remark}
Let $R$ be a Noetherian ring with a nonzero divisor $y$ contained in the Jacobson radical of $R$. It has been known if the residue ring $R/yR$ is an integrally closed domain, then $R$ itself must be an integrally closed domain (see for example \cite[Table 2]{Mu22}). A usual proof of this fact relies on some deep facts from the theory of Noetherian rings, such as Serre's normality criterion, thus not elementary. Here, we offer another proof, using \cite[Chapter V, \S 1.4, Proposition 15]{Bour98} whose proof utilizes the multiplicative nature of almost integrality.

Assume that $R$ is as above and consider the topology on $R$ induced by the $yR$-adic filtration. In this situation, \cite[Chapter III, \S 3.3, Proposition 3]{Bour98} gives that every principal ideal of $R$ is closed. Let $\gr_{yR}(R):=\bigoplus_{i=0}^\infty y^iR/y^{i+1}R$ be the associated graded ring. Since $y$ is a nonzero divisor, it follows that $\gr_{yR}(R) \cong R/yR[T]$, where $T$ is an indeterminate over $R/yR$ (for the proof of this isomorphism, see \cite[Theorem 16.2]{M86}). Since $R/yR$ is integrally closed, so is $\gr_{yR}(R)$. Since all rings involved are Noetherian, \cite[Chapter V, \S 1.4, Proposition 15]{Bour98} or \cite[Theorem 4.5.9]{BrHer93}
immediately gives that $R$ is integrally closed.
\end{remark}

\section{Appendix: Completion of valuation rings and Krull's work on complete integral closure}

We start with some background history. Let $p>0$ be a prime and let $\mathbb{Z}_p$ be the ring of $p$-adic numbers. We define $\mathcal{O}_{\mathbb{C}_p}$ to be the $p$-adic completion of the integral closure $\mathbb{Z}_p^+$ of $\mathbb{Z}_p$ in an algebraic closure of $\mathbb{Q}_p$. Then it is a well-known fact in algebraic number theory that $\mathbb{Z}_p^+$ and $\mathcal{O}_{\mathbb{C}_p}$ are valuation rings of dimension one. 

Our aim in this section is to extend this result to an arbitrary valuation ring of dimension one. Actually, such a result was already known to experts. For instance, there is a general result concerning the completion of valuation rings in \cite[Proposition 9.1.16]{GR22} and \cite[Proposition 6.2]{BM21}. However, our proof is motivated by Krull's early study on the complete integral closure of valuation rings. Indeed, we explain his classical result on the behavior of valuation rings under complete integral closure. The next proposition is due to Krull \cite{Krull31}. It seems that the conclusion of this proposition shows a particular distinctiveness of complete integral closure of general commutative rings. As the authors were not able to find it in some standard references or books, we decided to give a proof.

\begin{proposition}[Krull]
\label{completeintKrull}
Let $V$ be a valuation ring with field of fractions $K$. 
\begin{enumerate}
\item
If $V$ is of dimension one, then $V$ is completely integrally closed in $K$.

\item
If $V$ has a prime ideal of height one, then the complete integral closure of $V$ in $K$ is a valuation ring of rank one, which is equal to the localization $V_\fp$ for a height-one prime $\fp$. If $V$ does not have a prime ideal of height one, then the complete integral closure of $V$ is $K$.
\end{enumerate}
\end{proposition}

\begin{proof}
$(1)$: By the assumption that $V$ is $1$-dimensional, we have the attached non-archimedean norm $|\cdot|:K \to \mathbb{R}_{\ge 0}$. Then $V=\{a \in K~|~|a| \le 1\}$, and assume that $x \in K$ is almost integral over $V$. Then there is an element $0 \ne t \in V$ such that the sequence of elements $tx,tx^2,\ldots$ belongs to $V$. So we have $|t||x|^n=|tx^n| \le 1$. If $x \notin V$, then since $|x| >1$, it follows that $|tx^n| \to \infty$ as $n \to \infty$. This is a contradiction and hence, $x \in V$.

$(2)$: Recall that the set of ideals in a valuation ring is totally ordered, which we use below. Assume that $V$ has a height-one prime ideal $\fp$. Then we have a chain of inclusions: $V \subset V_\fp \subset K$ and $V_\fp$ is a valuation ring of rank one. Let $W$ be the complete integral closure of $V$ in $K$. Then since $W$ is an overring of the valuation ring $V$ and contained in $K$, it follows that $W$ is also a valuation ring. Let us prove that $W=V_\fp$, so that $W$ is a valuation ring of rank one. Since $V_\fp$ was shown to be a completely integrally closed domain in $(1)$, we see that $W \subset V_\fp$. Let $\fm$ be the maximal ideal of $V$, let $f \in \fm \setminus \fp$ be any element and let $t \in \fp$ be a nonzero element. We prove that $\frac{1}{f} \in W$. To this aim, consider $I:=\bigcap_{n>0} f^n V$. 

We prove that $I$ is a prime ideal. Let $a,b, \in V$ be such that $ab \in I$ and assume that $a,b \notin I$. Then we have
\begin{equation}
\label{fraction}
\frac{ab}{f^n} \in V;~\forall n>0.
\end{equation}
Moreover, as the set of ideals in a valuation ring is totally ordered and the hypothesis $a,b \notin I$, we can find some $m>0$ for which $\frac{f^m}{a},\frac{f^m}{b} \in V$. In particular,
$$
\frac{f^{m+1}}{a},\frac{f^{m+1}}{b} \in \fm
$$
and thus, $\frac{f^{2m+2}}{ab} \in \fm$. This together with $(\ref{fraction})$ implies $\frac{1}{f} \in \fm$. This is a contradiction to the choice of $f$. So we conclude that $a \in I$ or $b \in I$ holds and hence, $I$ is a prime ideal.

We have $\fp \subset I$, because if otherwise, we would have $\fp \not \subset I$. Then the only possibility that this happens is when $I=0$ and so $f^N \in \fp$ for $N \gg 0$ by the fact that the set of ideals are totally ordered in a valuation domain. But this is a contradiction. Now we have
$$
t \in \bigcap_{n>0} f^nV~\mbox{or equivalently},~t\Big(\frac{1}{f}\Big)^n \in V,
$$
which implies that $\frac{1}{f} \in K$ is almost integral over $V$ and belongs to $W$. So $W$ is a valuation ring of rank one, as desired.

In the case that $V$ does not have a prime ideal of height one, let $f \in \fm \setminus\{0\}$ be any nonzero element. Set $I=\bigcap_{n>0} f^nV$, which was shown to be a prime ideal. Suppose that $I=(0)$. Then it implies that $V$ is $f$-adically separated. By \cite[Chapter 0, Proposition 6.7.3]{FK18}, we see that $\sqrt{(f)}$ is a height-one prime, which is a contradiction. Hence we have that $I$ is a nonzero prime ideal and there is a nonzero element $t \in I=\bigcap_{n>0} f^nV$, from which it follows that $\frac{1}{f} \in K$ is almost integral over $V$, as desired.
\end{proof}

Using the above proposition, we prove the following result, which specializes to the case where $V=\mathcal{O}_{\overline{\mathbb{Q}}_p}$. See also \cite[Proposition 6.2]{BM21}.

\begin{proposition}
Let $V$ be a valuation ring of dimension one and let $t \in V$ be any element that is not zero and not a unit. Then the $t$-adic completion of $V$, denoted by $\widehat{V}$, is a valuation ring of dimension one. Moreover, the natural map $V \to \widehat{V}$ is injective.
\end{proposition}

\begin{proof}
Let $|\cdot|:K \to \mathbb{R}_{\ge 0}$ be the non-archimedean norm associated to $V$. Since $|t|<1$, we have $|t^n| \to 0$ $(n \to \infty)$. This implies that $V$ is $t$-adically separated. By \cite[Chapter 0, Theorem 9.1.1]{FK18}, the $t$-adic completion $\widehat{V}$ is a valuation ring, which in particular implies that $\widehat{V}$ is $t$-adically separated. By applying \cite[Proposition 6.7.2]{FK18}, we find that $\widehat{V}[\frac{1}{t}]$ is the field of fractions of $\widehat{V}$. By Proposition \ref{completeintKrull}, $V$ is completely integrally closed in $V[\frac{1}{t}]$. So \cite[Corollary 2.7]{NS19} shows that $\widehat{V}$ is completely integrally closed in $\widehat{V}[\frac{1}{t}]$. Again Proposition \ref{completeintKrull} applies and we conclude that $\widehat{V}$ is a valuation ring of dimension one.
\end{proof}

\end{document}